\documentclass[12pt]{amsart}

\usepackage[text={440.8pt, 576pt}, centering]{geometry}

\usepackage{amssymb}
\usepackage{bm}
\usepackage{paralist}

\usepackage{xcolor} 
\usepackage[colorlinks=true, urlcolor=violet, citecolor=blue]{hyperref}
\usepackage[capitalize]{cleveref} 

\newcommand{\excise}[1]{}

\newtheorem{thm}{Theorem}[section]
\newtheorem{lemma}[thm]{Lemma}
\newtheorem{cor}[thm]{Corollary}
\newtheorem{prop}[thm]{Proposition}

\newtheorem{prob}{Problem}
\newtheorem*{claim}{Claim}

\theoremstyle{definition}
\newtheorem{example}[thm]{Example}
\newtheorem{remark}[thm]{Remark}
\newtheorem{defn}[thm]{Definition}

\newtheorem{alg}[thm]{Algorithm}
\newtheorem{rtne}[thm]{Routine}

\newcounter{separated}

\newenvironment{alglist}%
    {\begin{list}
        {}
        {\leftmargin=4.8em\labelwidth=5.1em\labelsep=.2em
         \topsep=-1ex\itemsep=.1ex}\sf}  
    {\vspace{1ex}\end{list}}
\newcommand\init[1]{\item[{\bf{#1}:{\ }}]}
\newcommand\routine[1]{\item[{\sc{#1}{\ }}]}
\newcommand\procedure[1]{{\sc{#1}}}
\newenvironment{initlist}[1]%
    {\init{#1}\begin{list}
        {}
        {\leftmargin=0em\labelwidth=0em\labelsep=.5em
         \topsep=-1ex\itemsep=.1ex}}
    {\end{list}}
\newenvironment{routinelist}[1]%
    {\routine{#1}\begin{list}
        {}
        {\leftmargin=0em\labelwidth=0em\labelsep=.5em
         \topsep=-1ex\itemsep=.1ex}}
    {\end{list}}
    {\begin{list}
        {}
        {\leftmargin=3.0em\labelwidth=4.8em\labelsep=.5em
         \itemsep=.1ex\topsep=0ex}}
    {\end{list}}

\newcommand{\ring}[1]{\ensuremath{\mathbb{#1}}}

\renewcommand\>{\rangle}
\renewcommand\AA{\mathbb A}
\renewcommand\aa{\mathbf{a}}
\newcommand\<{\langle}

\newcommand\CC{\ring{C}}
\newcommand\II{\mathcal{I}}
\newcommand\NN{\ring{N}}
\newcommand\OO{\mathcal{O}}
\newcommand\PP{\ring{P}}

\newcommand\TT{A}
\newcommand\ZZ{\ring{Z}}
\newcommand\bb{\mathbf{b}}
\newcommand\cc{\mathbf{c}}
\newcommand\kk{\Bbbk}

\newcommand\pp{\mathfrak{p}}
\newcommand\xx{\mathbf{x}}

\newcommand\GL{\mathit{GL}}
\newcommand\NL{\mathit{NL}}
\renewcommand\th{\mathrm{th}}

\newcommand\too{\longrightarrow}
\newcommand\xxt{\mathbf{\tilde{x}}}
\newcommand\from{\leftarrow}
\newcommand\into{\hookrightarrow}

\newcommand\with{\mid}
\newcommand\minus{\smallsetminus}
\newcommand\nothing{\varnothing}

\renewcommand\implies{\Rightarrow}

\newcommand\ol[1]{{\overline {#1}}}

\newcommand\wt[1]{{\widetilde {#1}}}

\DeclareMathOperator{\Hom}{Hom} 
\DeclareMathOperator{\num}{\sf num} 
\DeclareMathOperator{\rank}{rank}
\DeclareMathOperator{\spec}{Spec}
\DeclareMathOperator{\Proj}{Proj}

\DeclareMathOperator{\Ende}{End}

\DeclareMathOperator{\ev}{ev} 
\DeclareMathOperator{\id}{id} 

\begin{document}

\mbox{}\vspace{-3.8ex}
\title{When is a polynomial ideal binomial after an ambient automorphism?}

\author{Lukas Katth{\"a}n}
\address{Institut f\"ur Mathematik, Goethe-Universit\"at\\ Frankfurt, Germany} 
\email{katthaen@math.uni-frankfurt.de}

\author{Mateusz Micha\l{}ek}
\address{Institute of Mathematics of Polish Academy of Sciences, Warsaw, Poland and
\newline
\mbox{}\,\,\,\,$\!\!\!\!$\quad
Max Planck Institute for Mathematics in the Sciences, Leipzig, Germany}
\email{wajcha2@poczta.onet.pl}

\author{Ezra Miller}
\address{Department of Mathematics\\Duke University\\Durham, NC 27708}
\urladdr{\url{http://math.duke.edu/people/ezra-miller}}

\date{11 June 2017}



\begin{abstract}
Can an ideal~$I$ in a polynomial ring $\kk[\xx]$ over a field be moved
by a change of coordinates into a position where it is generated by
binomials $\xx^\aa - \lambda\xx^\bb$ with $\lambda \in \kk$, or by
unital binomials (i.e., with $\lambda = 0$ or~$1$)?  Can a variety be
moved into a position where it is toric?  By fibering the
$G$-translates of~$I$ over an algebraic group~$G$ acting on affine
space, these problems are special cases of questions about a
family~$\II$ of ideals over an arbitrary base~$B$.  The main results
in this general setting are algorithms to find the locus of points
in~$B$ over which the fiber of~$\II$
\begin{itemize}
\item%
is contained in the fiber of a second family~$\II'$ of ideals
over~$B$;
\item%
defines a variety of dimension at least~$d$;
\item%
is generated by binomials; or
\item%
is generated by unital binomials.
\end{itemize}
A faster containment algorithm is also presented when the fibers
of~$\II$ are prime.  The big-fiber algorithm is probabilistic but
likely faster than known deterministic ones.\linebreak
Applications include the setting where a second group~$T$ acts on
affine space, in addition to~$G$, in which case algorithms compute the
set of $G$-translates of~$I$
\begin{itemize}
\item%
whose stabilizer subgroups in~$T$ have maximal dimension; or
\item%
that admit a faithful multigrading by~$\ZZ^r$ of maximal rank~$r$.
\end{itemize}
Even with no ambient group action given, the final application is an
algorithm to
\begin{itemize}
\item%
decide whether a normal projective variety is abstractly toric.
\end{itemize}
All of these loci in~$B$ and subsets of~$G$ are constructible; in some
cases they are closed.
\end{abstract}

\maketitle

\vspace{-1.8ex}
\section{Introduction}\label{s:intro}

\noindent
Ideals generated by binomials define schemes with much simpler
geometry than arbitrary polynomial ideals \cite{binomialideals},
largely yielding to analysis by combinatorial methods
\cite{primDecomp, mesoprimary, soccular}.  Similarly, ideals
homogeneous with respect to some grading or multigrading are simpler
than general ideals.  In principle, much of this simplicity persists
after linear change of coordinates, or any other automorphism of the
ambient affine space.  Therefore, it seems natural to ask for an
algorithm to decide whether a given ideal has any of these properties
after applying an ambient automorphism.

Here we present algorithms for this and related tasks.  Let $G$ be an
algebraic group acting on affine space $\AA^n =
\spec\kk[x_1,\dots,x_n]$ via a morphism $\alpha: G \times \AA^n \to
\AA^n$.  For a given ideal $I \subseteq \kk[\xx] =
\kk[x_1,\dots,x_n]$, we provide algorithms for the following tasks,
where $I.\gamma$ means the image of under the corresponding (right)
action of $\gamma \in G$ on~$\kk[\xx]$.
\begin{enumerate}[({\TT}1)]
\item\label{A:bin}%
Find the elements $\gamma \in G$ such that $I.\gamma$ is generated by
binomials (Algorithm~\ref{a:binomial}).

\item\label{A:unit}%
Find the elements $\gamma \in G$ such that $I.\gamma$ is unital,
meaning generated by monomials and differences of monomials
(Algorithm~\ref{a:checkUnital}).

\item\label{A:grad}%
Find the elements $\gamma \in G$ such that $I.\gamma$ admits a
faithful $\ZZ^r$-grading with $r$ as large as possible
(Algorithm~\ref{a:grading}).

\item\label{A:group}%
Given a second algebraic group $T$ which also acts on $\AA^n$, find
the elements $\gamma \in G$ such that $I.\gamma$ is stabilized by as
large a subgroup of $T$ as possible (Algorithm~\ref{a:group}).
\end{enumerate}\setcounter{separated}{\value{enumi}}
In each item, we can in particular determine whether there exists
$\gamma \in G$ such that $I.\gamma$ has the respective property.

The case of $G = \GL_n$ in~(A\ref{A:bin}) originated with Eisenbud and
Sturmfels \cite[page~6]{binomialideals}, who raised the issue of
determining when a given ideal is the image of a binomial ideal under
an ambient linear automorphism.

It would be desirable to have algorithms for these questions with $G$
being the entire automorphism group of affine space.  However, as this
group is currently a mystery \cite{polyAut, SU04}, it seems only fair
to require that $G$ be specified beforehand.  In fact, the reader will
lose none of the flavor or difficulty by assuming that $G = \GL_n$
consists of all linear changes of coordinates.

The most important class of binomial ideals consists of those that are
\emph{toric}, meaning unital and prime.  They define affine toric
varieties, so (A\ref{A:unit}) can be used to find automorphims in~$G$
under which the image of a given affine variety is (equivariantly
embedded as) a toric variety.  For varieties that are projective, we
can even do better and detect whether they are toric without having to
specify a group $G$ beforehand:
\begin{enumerate}[({\TT}1)]\setcounter{enumi}{\value{separated}}
\item\label{A:toric}%
Given a normal, projective variety $X \subseteq \PP^n$, decide if it
is toric (Algorithm~\ref{a:checkProjToric}).
\end{enumerate}\setcounter{separated}{\value{enumi}}
Even if it is not toric, taking the group~$T$ in~(A\ref{A:group}) to
be the algebraic torus~$(\kk^*)^n$, our method finds a large subtorus
acting on $V(I.\gamma)$, turning the latter optimally into a
$T$-variety.  (If your definition of $T$-variety requires normality,
then of course it can only work if $V(I)$ is normal to begin with.)

Here are two examples where ``hidden'' toric structures turned out to
be useful.

\begin{example}[Phylogenetics and group-based models]
Group-based models are special statistical models: maps from the
parameter space to the space of probability distributions
\cite{Phylogeneticalgebraicgeometry}.  In their original coordinates
the maps are not monomial, so the Zariski closures of the images do
not seem toric.  However, a clever linear change of coordinates, known
as the Discrete Fourier Transform, turns the varieties to
equivariantly embedded toric varieties \cite{hendy1989framework,
sturmfels2005toric}. This fact inspired numerous mathematicians both
in statistics and in algebraic geometry \cite{buczynska2007geometry,
draisma2009ideals, casanellas2011relevant, michalek2016finite}.
\end{example}

\begin{example}[Secant and tangential varieties of Segre--Veronese]
Secant and tangential varieties are classical topics in algebraic
geometry \cite{Zak}.  As an example of the difficulty of their
geometric and algebraic properties, finding the defining equations of
the secant variety of any Segre--Veronese variety was an open
conjecture of Garcia, Stillman, and Sturmfels
\cite{garcia2005algebraic}, solved only recently by Raicu
\cite{raicu2012secant}.  Thus, it is surprising that both secant and
tangential varieties of Segre--Veronese are covered by open toric
varieties---complements of hyperplane sections
\cite{sturmfels2013binary, MOZ, MPS}.  Here a nonlinear change of
coordinates, inspired by computation of cumulants in statistics
\cite{zwiernik2015semialgebraic} played a crucial role.
\end{example}

\subsection*{Methods}\label{sub:methods}

The principle that guides our algorithms concerns the comparison of
families parametrized over a common base.  When $G$ acts on~$\AA^n$,
the $G$-translates of~$I$ fiber over~$G$ (Definition~\ref{d:orbit} and
Remark~\ref{r:orbit}).  If a second group~$T$ acts on~$\AA^n$ and it
is desired to find a subgroup of~$T$ that stabilizes a $G$-translate of
$X \subseteq \AA^n$, then ask for $\tau \in T$ and $\gamma \in G$ such
that $\gamma.X$ is stabilized by~$\tau$.  This problem fibers over $T
\times G$ (Section~\ref{sub:big-actions}), the point being to find the
locus $Y \subseteq T \times G$ over which $\tau.(\gamma.X) =
\gamma.X$.  In the context of~(A\ref{A:group}), where the goal is to
move~$X$ so as to make its stabilizing subgroup as large as possible,
the algorithm then finds the locus of points in~$G$ over which the
fiber of $Y \to G$ has maximal dimension.  This locus is closed
(Proposition~\ref{p:largegroup}).  Note that Algorithm~\ref{a:grading}
for~(A\ref{A:grad}) is the special case where $T = (\kk^*)^n$ is the
algebraic torus acting diagonally on~$\AA^n$
(Section~\ref{sub:multigradings}).

The upshot is that our computational engine consists of two algorithms
for a family of ideals over an arbitrary base~$B$: find the locus of
points in~$B$ over which the fiber
\begin{enumerate}[({\TT}1)]\setcounter{enumi}{\value{separated}}
\item\label{A:contain}%
is contained in the fiber of a second given family of ideals over~$B$
(Algorithm~\ref{a:equalfibers})~or
\item\label{A:dim}%
defines a variety of dimension at least~$d$ (Algorithm~\ref{a:dim}).
\end{enumerate}
These rely on geometry of constructible sets
(Section~\ref{sub:constructible}) and a bit of flatness, when it
is desired that the constructible outputs of the algorithms be closed
(Theorem~\ref{t:equal-locus2}).  Deterministic algorithms
for~(A\ref{A:dim}) are known \cite{kemper}, but Algorithm~\ref{a:dim}
is probabilistic and likely faster.

Having already the context of an arbitrary base~$B$ at our disposal,
the algorithms for~(A\ref{A:bin}) and~(A\ref{A:unit}) work as well for
a family of ideals over~$B$.
\begin{enumerate}[({\TT}1$'$)]
\item[(A\ref{A:bin}$'$)]%
Find the locus of points in~$B$ over which the fiber is binomial
(Algorithm~\ref{a:binomial}).

\item[(A\ref{A:unit}$'$)]%
Find the locus of points in~$B$ over which the fiber is unital
(Algorithm~\ref{a:checkUnital}).
\end{enumerate}
The reason is that the criterion for binomiality simply detects
whether the reduced Gr\"obner basis is binomial
(Section~\ref{sub:alg-binom}); it has nothing to do with a group
action on the ambient affine space.  Similarly, an ideal is unital
precisely when its scheme is closed under coordinatewise
multiplication (Proposition~\ref{p:monoidal}); while this employs the
monoid structure on~$\kk^n$, which is defined only once a basis of
$\kk^n$ has been given, that monoid action is fixed from the outset,
so it remains only to calculate which fibers respect it.

The final application, (A\ref{A:toric}), observes that high Veronese
embeddings of normal projective varieties are projectively normal,
after which an ambient automorphism must make the variety toric if
anything can (Theorem~\ref{t:checkProjToric}).  Thus our method
applies~(A\ref{A:unit}) with $G$ being the general linear group acting
on projective space.  A speedup (Algorithm~\ref{a:equalfibersIrreduc})
for~(A\ref{A:contain}) is available in this case because the fibers
are known to be prime.

\subsection*{Conventions}

Everything throughout the paper is over a field~$\kk$.  Implementation
of the algorithms would require that $\kk$ be ``computable'' in some
appropriate sense, but most of our discussions are independent of this
hypothesis.  That said, we do assume, without further comment, that
all algebras and schemes are of finite type over~$\kk$.

Functions, variables, and spaces are denoted by English letters.
Group elements or points in spaces are denoted by Greek letters.  Thus
$f(\xx)$ for $\xx = x_1,\dots,x_n$ is a function on a subscheme $X$
of affine $n$-space $\AA^n$ whose action on a $\kk$-valued point $\xi
\in X$ is $\xi \mapsto f(\xi)$.  Hopefully this eliminates confusion
regarding left vs.\ right actions on spaces vs.\ functions, the
details of which are covered in Section~\ref{sub:actions}.  When we
wish to think of points algebraically, as prime ideals in rings, then
we use Fraktur letters.  Thus $\pp_\gamma \subseteq \kk[G]$ is the
prime ideal corresponding to the point $\gamma \in G$.  For any prime
ideal~$\pp$, its residue field is written~$\kappa_\pp$.  If $\pp =
\pp_\xi$, say, then we also write $\kappa_\xi = \kappa_{\pp_\xi}$.

\subsection*{Acknowledgments}

We are grateful to the Mathematical Society of Japan for hosting its
8th Seasonal Institute in Osaka, ``Current trends on Gr\"obner bases''
and inviting two of us (MM and EM) in July, 2015; this work originated
with questions raised there in discussions with Thomas Kahle.  We
thank the Italian Istituto Nazionale di Alta Matematica for sponsoring
the meeting on ``Homological and computational methods in commutative
algebra'' in honor of Winfried Bruns's 70th birthday in Cortona, June
2016; that remains the sole instance when all three of us were
physically in the same location.  We are grateful to the European
Mathematical Society and the Foundation Compositio Mathematica for
supporting the 24th National School on Algebra and EMS Summer School
on Multigraded Algebra and Applications, where two of us (LK and EM)
had extended discussions.  Further, we thank the Research Institute
for Mathematical Sciences (RIMS) and the Kyoto University for hosting
the meeting on ``Computational Commutative Algebra and Convex
Polytopes'' in August 2016, which again gave two of us (LK and MM) the
opportunity for discussions.  EM is grateful to the Max Planck
Institute f\"ur Mathematik in den Naturwissenschaften in Leipzig,
Germany for funding a stay that largely pushed this work to
conclusion.  MM was supported by Polish National Science Centre grant
no.\thinspace{}UMO-2016/22/E/ST1/00574 and the Foundation for Polish
Science (FNP). LK was funded by the German Research Foundation (DFG),
grant no.\thinspace{}KA 4128/2-1.

\section{Algorithms for families of schemes}\label{s:algs}

Generally speaking, our algorithms are aimed at schemes over groups,
thought of as families of schemes (or ideals) parametrized by the
group.  But many of our results hold over more arbitrary base schemes;
we phrase those in terms of a commutative $\kk$-algebra~$S$.  The
polynomial ring $S[\xx]$, where $\xx = x_1,\dots,x_n$ denotes the
sequence of variables, has spectrum $\AA^n_S = \AA^n \times \spec S$,
the affine space of dimension~$n$ over (the spectrum of)~$S$.  An
ideal $J \subseteq S[\xx]$ corresponds to a subscheme $X \subseteq
\AA^n_S$, usually thought of as a family of subschemes of~$\AA^n =
\AA^n_\kk$ parametrized by (the $\kk$-valued points of)~$\spec S$, or
as a family of ideals of~$\kk[\xx]$.  But if $\pp$ is any prime ideal
of~$S$, maximal or not, then the ideal defining the fiber $X_\pp
\subseteq \AA^n_\pp$ over~$\pp$ is a specialization of~$J$, namely the
extension $J\kappa_\pp[\xx]$ of~$J$ to the polynomial ring over the
residue field $\kappa_\pp = S_\pp/\pp S_\pp$~of~$\pp$.

\subsection{Constructible sets}\label{sub:constructible}

Let $S$ be a commutative $\kk$-algebra.  A subset of $\spec S$ is
\emph{constructible} if it is a finite union
$$%
  \bigcup_{i = 1}^s U_i \cap C_i,
$$
where each $U_i \subseteq \spec S$ is open and each $C_i \subseteq
\spec S$ is closed.  We assume access to algorithms that
\begin{enumerate}
\item%
compute unions and intersections of constructible sets, and

\item%
determine whether a constructible set is empty.
\end{enumerate}
See \cite{brunat-montes2016}, for example.

The following is used in the proof of Theorem~\ref{t:equal-locus2}.

\begin{lemma}\label{l:linal}
Let $S$ be a commutative $\kk$-algebra.  For any $n \times m$-matrix
$A$ with entries in~$S$ and column vector $b \in S^m$, the set of
primes $\pp \subseteq \spec S$ such that the system $A x = b$ of
linear equations has a solution over the residue field $\kappa_\pp$ is
constructible in $\spec S$.  If $A$ has the same rank over every prime
$\pp \in \spec S$, then this set is closed.
\end{lemma}
\begin{proof}
Let $A'$ be the matrix obtained by appending the column $b$ to $A$.
For fixed $\pp$, the equation $A x = b$ is solvable over $\kappa_\pp$
if and only if $A'$ has the same rank as $A$ over~$\kappa_\pp$.
	
Let $U_i$ and $U_i'$ be the subsets of $\spec S$ where $A$ and~$A'$
have rank at most~$i$, respectively.  As $\rank A' \geq \rank A$ in
any case, the set of primes where the system is~solvable~is
$$%
  \bigcup_i U_i' \minus U_{i-1}.
$$
Each $U_i$ and $U_i'$ is closed, being the zero set of some minors, so
this set is~\mbox{constructible}.
	
Finally, if the rank of~$A$ is constant, say $\rank(A) = r$, then the
system is solvable exactly over~$U'_r$, which is closed.
\end{proof}

\begin{example}\label{e:linear-not-closed}
The linear equation $t \cdot x = 1$ over $S = \kk[t]$ is solvable if
and only if $t \neq 0$, so this locus need not be closed when the rank
of~$A$ in Lemma~\ref{l:linal} varies.
\end{example}

For the sake of completeness we recall an algebraic version of Baire's
theorem needed in the proof of Lemma~\ref{l:construct}, which is in
turn used in the proof of Proposition~\ref{p:finiteunion}.

\begin{thm}[Baire's Theorem]\label{t:Baire}
No irreducible scheme~$X$ of finite type over an uncountable
field~$\kk$ is, as a topological space, a countable union of closed
proper subsets.
\end{thm}
\begin{proof}
As only the topology is in play, there is no harm in assuming that all
of the schemes involved are reduced.  Intersecting each subset in such
a union $\bigcup_{i=1}^\infty Z_i$ with the members of an affine open
cover reduces to the case where $X$ is affine.  Hence $X$ is a closed
subscheme of~$\AA^n$, and its coordinate ring is a domain because it
is reduced and irreducible.  Noether normalization therefore reduces
to the case $X = \AA^n$, because it guarantees a finite surjective
morphism to~$\AA^{\dim X}$ while preserving the fact that each $Z_i$
is a closed proper subset.

The goal is to show that $\bigcup_{i=1}^\infty Z_i$ is a proper subset
of~$\AA^n$ given that each~$Z_i$ is a proper closed subset.  The proof
is by induction on~$n$, the case $n = 1$ following from the
uncountability of~$\kk$.  Let $H$ be a hyperplane that contains none
of the (countably many) irreducible components of the~$Z_i$; such
an~$H$ exists---simplest is to choose it parallel to some given
hyperplane---because $\kk$ is uncountable.  The induction is concluded
by noting that $H \cap Z_i$ is a proper closed subscheme of~$H$ for
all~$i$, so $H \neq \bigcup_{i=1}^\infty H \cap Z_i$.
\end{proof}

\begin{lemma}\label{l:construct}
If the field~$\kk$ is uncountable, then no countably infinite union of
disjoint nonempty constructible sets is constructible.
\end{lemma}
\begin{proof}
Let $U = \bigcup_{i=1}^\infty Z_i$ be the union of disjoint nonempty
constructible sets~$Z_i$.  The proof is inductive on $\dim \ol U$, the
case $\dim \ol U = 0$ being trivial.

For contradiction, suppose $U$ is a finite union of intersections $C_j
\cap O_j$ of closed irreducible sets~$C_j$ and open sets~$O_j$.  One
of the intersections, say $C_1 \cap O_1$, must intersect infinitely
many~$Z_i$ nontrivially.  Let $Z_i' := Z_i \cap C_1$.  Then \mbox{$C_1
= (C_1 \minus O_1) \cup \bigcup_{i=1}^\infty Z_i'$}.
Theorem~\ref{t:Baire} implies that one of the sets~$Z_i'$ must
equal~$C_1$.  But this contradicts the hypothesis that the $Z_i$ are
disjoint and infinitely many of them intersect~$C_1$.
\end{proof}

\subsection{The locus of fiber containment}\label{s:locus}

For two schemes fibered over a fixed base scheme~$B$, the methods in
later sections rely on an ability to compute the locus of points
in~$B$ where the fibers of the first scheme are contained in the
fibers of the second one.  In an affine setting, this locus is
described by the Theorem~\ref{t:equal-locus2}.  For terminology, a
ring is \emph{connected} if it has no nontrivial idempotents---or
equivalently, if its spectrum is connected.  For a polynomial $f$ of
total degree~$d$ in variables~$\xx = x_1,\dots,x_n$, its
\emph{homogenization} is the homogeneous polynomial
$$%
  \wt f = x_0^d f\Big(\frac{x_1}{x_0},\dots,\frac{x_n}{x_0}\Big)
$$
of degree~$d$ in $\xxt = x_0,\dots,x_n$ that yields $f$ when $x_0$ is
set equal to~$1$.  The \emph{homogenization} of an ideal~$I$ in a
polynomial ring with variables~$\xx$ is the ideal $\wt I$ in the
polynomial ring in $\xxt$ generated by $\{\wt f \with f \in I\}$.

\begin{thm}\label{t:equal-locus2}
Let $S$ be a connected commutative $\kk$-algebra.  If $I_1, I_2
\subseteq S[\xx]$ are two ideals, then the set of points $\{\pp \in
\spec S \with I_1\kappa_\pp[\xx] \subseteq I_2\kappa_\pp[\xx]\}$ is
constructible in $\spec S$.  If the quotient $S[\xxt]/\wt{I}_2$ by the
homogenization $\wt I_2$ is flat over $S$, then this set is closed.
\end{thm}
\begin{proof}
Let $f_1, \dots, f_r$ generate~$I_1$.  Then $I_1\kappa_\pp \subseteq
I_2\kappa_\pp$ if and only if $I_1 S_\pp[\xx] + \pp S_\pp[\xx]
\subseteq I_2 S_\pp[\xx] + \pp S_\pp[\xx]$, and this is in turn
equivalent to $f_i \in I_2 S_\pp[\xx] + \pp S_\pp[\xx]$ for all $i$.
Therefore it suffices to treat the case where $I_1$ is principal, say
$I_1 = \<f\>$.

Observe that $f \in I_2S_\pp[\xx] + \pp S_\pp[\xx]$ if and only if
$\wt f \in \wt I_2 S_\pp[\xxt] + \pp S_\pp[\xxt]$.
Suppose that $\wt I_2$ is generated by the homogeneous polynomials
$q_1, \dots, q_r \in S[\xxt]$ with degrees $d_1,\dots,d_r$, and let $d
= \deg f$.  Consider the ansatz:
\begin{equation}\label{eq:formal2}
  \wt f = \sum_{i=1}^r \sum_{\substack{\aa \in \NN^{n+1}
\\
  |\aa| = d - d_i}} c_{i, \aa} \xx^\aa q_i.
\end{equation}
The coefficients $c_{i, \aa}$ are considered as new unknowns, so
\Cref{eq:formal2} can be regarded as a system of linear equations for
the $c_{i, \aa}$ with coefficients in $S$.  By construction, $\wt f
\in \wt I_2 S_\pp[\xxt] + \pp S_\pp[\xxt]$ if and only if this system
has a solution modulo~$\pp$.  Hence, the desired constructibility
follows from Lemma~\ref{l:linal}.
	
Now assume that $S[\xxt]/\wt{I}_2$ is flat over $S$. Then the Hilbert
function of $S[\xxt]/\wt{I}_2 \otimes_S \kappa_\pp =
\kappa_\pp[\xxt]/\wt I_2 \kappa_\pp$ is locally constant as a function
of~$\pp$ \cite[Ex 20.14]{Eis}, so it does not depend on~$\pp$ at all
because $S$ is connected.  The short exact sequence
$$%
  0 \too \wt I_2\kappa_\pp \too \kappa_\pp[\xxt] \too
  \kappa_\pp[\xxt]/\wt I_2 \kappa_\pp \too 0
$$
implies that the Hilbert function of $\wt{I}_2 \kappa_\pp$ also does not
depend on~$\pp$.  But the value of this Hilbert function in degree~$d$
is the rank of the matrix on right-hand side of \Cref{eq:formal2}
at~$\pp$.  Hence the closedness claim follows from the last part of
Lemma~\ref{l:linal}.
\end{proof}

\begin{defn}\label{d:locus}
The constructible set in Theorem~\ref{t:equal-locus2} is the
\emph{containment locus} for~$I_1$ in~$I_2$ (or
for~$X_2$ in~$X_1$, where $I_i = I(X_i)$).  Its intersection with the
containment locus for $I_2$ in~$I_1$ is the \emph{coincidence locus}
of~$I_1$ and~$I_2$ (or of~$X_1$ and~$X_2$).
\end{defn}

\begin{example}
The loci considered in Theorem~\ref{t:equal-locus2} need not be
closed.  Indeed, let $S = \kk[s,t]$ and $I_1, I_2 \subseteq S[x]$ be
defined by $I_1 = \<x\>$ and $I_2 = \<sx - t\>$.  Geometrically,
$V(I_1)$ is just the $st$-plane, while $V(I_2)$ is the affine part of
a blow-up of this plane at the origin.  The fibers coincide exactly
when $t = 0$ and $s \neq 0$.
\end{example}

\begin{example}\label{e:non-flat-homogenization}
The flatness hypothesis in Theorem~\ref{t:equal-locus2} is on $\wt
I_2$ and not simply on~$I_2$ because flatness of $S[\xx]/I$
over~$S$ does not imply that the homogenization~$S[\xxt]/\wt{I}$ is flat
over~$S$.  Take $S = \kk[a,b]$ and $I = \<ax - 1, by - 1\> \subseteq
S[x,y]$.  Then $S[x,y]/I \cong S[a^{-1},b^{-1}]$ is flat over~$S$
because it is a localization of~$S$.  But the homogenization of~$I$ is
$\wt I = \<ax - z, by - z\>$, and $S[x,y,z]/\wt I = \kk[a,b,x,y]/\<ax
- by\>$, which fails to be flat over the $ab$-plane for the same
reason that $(a,b)$ fails to be a regular sequence.
\end{example}

We now present an algorithmic version of Theorem~\ref{t:equal-locus2}.

\begin{alg}\label{a:equalfibers}
Compute the containment locus for two families over same base
\end{alg}
\begin{alglist}
\init{Input} ideals $I_1,I_2 \subseteq S[\xx]$ of families over a
	commutative $\kk$-algebra~$S$

\init{Output} containment locus for $I_1$ in~$I_2$ as
	$\bigcap_{i=1}^k \bigcup_{j=1}^{\ell_i} \big(V(J'_{ij}) \minus
	V(J_{ij})\big)$ for $J'_{ij}, J_{ij} \subseteq S$

\begin{routinelist}{define}
\item[] $f_1,\dots,f_k$ generators of~$I_1$, with homogenizations
	$\wt f_1, \dots, \wt f_k$
\item[] $q_1,\dots,q_r$ generators of the homogenization~$\wt I_2$,
	with degrees $d_1,\dots,d_r$
\end{routinelist}

\routine{while} $1 \leq i \leq k$ \procedure{do}

\begin{routinelist}{define}
\item[] $d := \deg f_i$
\item[] variables $c_{j,\aa}$ for $j \in \{1,\dots,r\}$ and $\{\aa \in
	\NN^{n+1} \,\big|\, |\aa| = d - d_j\}$
\item[] $L$, the system $\wt f_i = \sum_{j=1}^r \sum_{\substack{\aa \in
	\NN^{n+1}\\ |\aa| = d - d_j}} c_{j,\aa} \xx^\aa q_j$ of linear
	equations
     \begin{itemize}
     \item indexed by monomials of degree~$d$ and
     \item with coefficients in~$S$
     \end{itemize}
\item[] matrix $A$ with coefficients in~$S$ to represent~$L$
\item[] vector~$b$ with coefficients in~$S$ to represent~$\wt f_i$ (so
	$L$ is given by $A\cc = b$)
\item[] $A' := {[A|b]}$, the matrix obtained by appending the
	column~$b$ to~$A$
\item[] $J'_{ij} :=$ ideal of size~$j$ minors of~$A'$, for $1 \leq j
	\leq \ell_i = $ size of~$A'$
\item[] $J_{ij} := $ ideal of size $j-1$ minors of~$A$ for the same
	values of~$j$
\item[] $Z_i := \bigcup_j \big(V(J'_{ij}) \minus V(J_{ij})\big)$
\end{routinelist}

\routine{end}{}\procedure{while-do}\vspace{.5ex}

\routine{return} $\displaystyle Z = \bigcap_{i=1}^k Z_i$ defined by
	the ideals $J'_{ij}, J_{ij} \subseteq S$ as $\displaystyle Z_i
	= \bigcup_j \big(V(J'_{ij}) \minus V(J_{ij})\big)$
\end{alglist}

\subsection{The locus of large fibers}\label{sub:locus}

Given a morphism of schemes $X \to B$, we need to compute the locus
$B^{\geq d} \subset B$ of fibers of dimension at least~$d$.  A
deterministic algorithm for this task was given by Kemper
\cite{kemper}.  The probabilistic algorithm presented here is
therefore not theoretically required for algorithms in the rest of
paper, but it is much easier to implement and we expect it to run
faster.  In particular, Algorithm~\ref{a:dim} has allowed us to
compute explicit examples, some of which are presented here.

Note on conventions: Algorithm~\ref{a:dim} and its proof assume (and
implicitly use) that the field~$\kk$ is algebraically closed.  The
engine of the algorithm---and the only interesting part---is the following
subroutine that takes a scheme affine over the base as input; it
inherits the hypotheses from its parent algorithm where it is applied.

\begin{rtne}\label{r:dim}
(\procedure{affine case of fiber dim $\geq d$})
\end{rtne}
\begin{alglist}
\init{Input} affine morphism $f:X \to B$ of schemes, say $X
\subseteq \AA^n_B$

\init{Output} closure $B^{\geq d}$ of the locus of points in~$B$
	over which the fiber has dimension $\geq d$

\begin{routinelist}{initialize}
\item[] $B_0 := \nothing$
\item[] $B_1 := B$
\item[] $i := 1$
\end{routinelist}

\routine{while} $B_{i-1} \neq B_i$ \procedure{do}

\routine{choose} random affine subspace $L' \subseteq \AA_\kk^n$ of
	codimension~$d$

\begin{routinelist}{define}
\item[] $P :=$ closure of the projection to $B$ of $L' \cap X$
\item[] $B_{i+1} := B_i \cap P$
\end{routinelist}

\routine{advance} $i \from i+1$

\routine{end}{\procedure{while-do}}

\routine{return} $B_i$
\end{alglist}

\pagebreak

\begin{alg}\label{a:dim}
Find locus of big fiber dimension
\end{alg}
\begin{alglist}
\init{Input} integer $d$ and morphism $f:X \to B$ of schemes over
	algebraically closed field~$\kk$

\init{Output} closure $B^{\geq d}$ of the locus of points in~$B$
	over which the fiber has dimension $\geq d$

\routine{compute} open cover $X = \bigcup_{j=1}^k X_j$ by
	subschemes~$X_j$ affine/$B$ (note: $X_j$ affine suffices)

\routine{define} $B^{\leq d}_j := $ \procedure{affine case of fiber
	dim $\geq d$} applied to~$X_j$ for~\mbox{$j = 1,\dots,k$}

\routine{return} $B^{\leq d}_1 \cup \dots \cup B^{\leq d}_k$
\end{alglist}

\begin{prop}\label{p:dim}
Algorithm~\ref{a:dim} is correct.
\end{prop}
\begin{proof}
It suffices to prove Routine~\ref{r:dim} works on an affine morphism
because a fiber is large if and only if it is large in (at least) one
member of an open cover of the source.  So assume $X \subseteq
\AA^n_B$.  Fix an irreducible component $Z$ of $B^{\geq d}$.
\begin{claim}
For a general point $z\in Z$ and a general affine subspace $L
\subseteq \AA^n$ of codimension~$d$, the fiber $f^{-1}(z)$ intersects
$L$.
\end{claim}
\begin{proof}[Proof of the claim]
This is deduced by taking closures in projective space, where
subvarieties of complementary dimension always intersect.  Of course
the projective closure of~$L$ is only guaranteed to intersect the
projective closure of the fiber, but the difference between an affine
algebraic subset and its projective closure has lower dimension.
\end{proof}

\noindent
The claim implies that $L' \cap X$ in Routine~\ref{r:dim} contains
general points of~$Z$, so $B^{\geq d} \subseteq B_i$.

It remains to prove that the output $B_i$ cannot be strictly bigger.
Consider a component $Q$ of any~$B_i$ (not necessarily the output
one).  Assume the fiber over a general point of $Q$ has dimension
smaller than~$d$.  To finish, we have to prove that $B_{i+1} \neq
B_i$.  Consider a fiber $F$ over a general point $q \in Q$.  Since
$\dim F < d$, a general subspace of codimension~$d$ in $\PP^n$ fails
to meet the closure of~$F$ in~$\PP^n$, and this remains true in a
neighborhood of $q$.  Thus the general point of~$Q$ is not in the
projection of $L' \cap X$.  Hence $Q$ is not a component of~$B_{i+1}$.
\end{proof}

\begin{remark}\label{r:radical}
The algorithms in this paper do not assume that constructible sets are
presented using radical ideals.  In particular, the algorithm outputs
might not be radical ideals.  However, when comparing two
constructible sets to discover containment or equality, it is
sometimes simplest to take radicals.  In the case of
Routine~\ref{r:dim}, radicals could be used (this occurs explicitly in
Example~\ref{e:macaulay2}, which uses Algorithm~\ref{a:dim}); but for
us it was faster to first run the subroutine a few times---with
different linear forms---without computing radicals, and only later
compare sets by computing radicals.
\end{remark}

\section{Detecting big group actions and multigradings}\label{sec:T-act}

\subsection{Group actions and families}\label{sub:actions}

Let $G$ be an algebraic group (over~$\kk$, as always), acting on
affine $n$-space $\AA^n$ over~$\kk$.  The action is a morphism
$\alpha: G \times \AA^n \to \AA^n$, and it is assumed to be a left
action, so $\gamma.(\gamma'.\xi) = (\gamma\gamma').\xi$ for points
$\gamma,\gamma' \in G$ \mbox{and $\xi \in \AA^n$ over~$\kk$}.

The geometric action on~$\AA^n$ is equivalent to an algebraic action
on~$\kk[\xx]$; for $\gamma \in G$ and $f \in \kk[\xx]$, the function
$f.\gamma$ sends $\xi \mapsto f(\gamma.\xi)$ for all $\xi \in \AA^n$.
More formally, $\alpha$ induces a ring homomorphism $\alpha^*:
\kk[\xx] \to \kk[G] \otimes \kk[\xx]$, where $\kk[G]$ is the
coordinate ring of $G$ (not its group algebra over~$\kk$), satisfying
the axioms dual to the group action axioms.  For any point $\gamma \in
G$ let $\ev_\gamma: \kk[G] \to \kappa_\gamma$ be the evaluation map,
meaning the algebra morphism corresponding to the inclusion
$\{\gamma\} \into G$.  Then for a $\kk$-valued group element~$\gamma$,
the composition $(\ev_\gamma \otimes \id) \circ \alpha^*$ gives a map
$\alpha^\dagger: G \to \Ende(\kk[\xx],\kk[\xx])$, and $f.\gamma =
\alpha^\dagger(\gamma)(f)$.  For any ideal $I \subseteq \kk[\xx]$ and
point $\gamma \in G$ over~$\kk$, set $I.\gamma = \{f.\gamma \with f
\in I\}$.

In the setting of schemes over~$G$, the crucial families are the
orbits.  To define them, one more bit of general notation helps: for
any two schemes $X$ and~$B$ over~$\kk$, write $X_B = B \times X$ and
consider it as a family over~$B$.  The reader will lose little by
thinking always of $B = G$, as in the following definition.

\begin{defn}\label{d:orbit}
The \emph{orbit morphism} is $\omega := \pi_G \times \alpha: \AA^n_G
\to \AA^n_G$, where $\pi_G$ is the projection to~$G$.  If $X \subseteq
\AA^n$ is a subscheme, then its \emph{orbit} is $\OO_G X =
\omega(X_G)$.
\end{defn}

\begin{remark}\label{r:orbit}
On $\kk$-valued points, the orbit morphism is $(\gamma,\xi) \mapsto
(\gamma,\gamma.\xi)$.  Geometrically, the orbit of~$X$ is a family
over~$G$ whose fibers are the translates of~$X$ by group elements.
The terminology comes from the case where $X$ is a point~$\xi$,
because the projection of $\OO_G \xi$ to~$\AA^n$ is indeed the
$G$-orbit of~$\xi$.

More formally, if $X \subseteq \AA^n$ is a subscheme, then $\gamma.X$
is the fiber of $\OO_G X$ over $\gamma \in G$.  Algebraically, $X$ is
defined by an ideal $I \subseteq \kk[\xx]$, and $\OO_G X$ is defined
by the ideal $J \subseteq S[\xx]$, where $S = \kk[G]$ and $J =
(\omega^*)^{-1}(I)$.  The ideal $I.\gamma^{-1}$ defining the subscheme
$\gamma.X$ of the affine space $\AA^n_\gamma$ over the residue
field~$\kappa_\gamma$ is a specialization of~$J$, namely the extension
$J\kappa_\gamma[\xx]$ of~$J$ to the polynomial ring over the residue
field of~$\gamma$.
\end{remark}

\begin{remark}\label{r:inverse}
Images of ideals are much easier to compute than preimages.
Therefore, in order to compute the ideal $J$ above in a concrete case,
it might be a good idea to first compute the inverse map of $\omega$,
which amounts to computing the map $\gamma \mapsto \gamma^{-1}$ on
$G$, and then obtain $J$ as push-forward of $I$ along that map. 
\end{remark}

The locus where an orbit is contained in any given family is always
closed.

\begin{cor}\label{c:constantclosed}
If a subscheme $X \subseteq \AA^n$ is given and $X_2$ is the constant
family~$X_G$ or orbit~$\OO_G X$, then the containment locus of~$X_2$
in~$X_1$ is closed for any family~$X_1$ over~$G$.
\end{cor}
\begin{proof}
The homogenization of a trivial family is trivial and hence flat;
therefore the $X_G$ case follows immediately from
Theorem~\ref{t:equal-locus2}.  The $\OO_G X$ case follows by first
applying the inverse $\omega^{-1}$ of the orbit morphism, then
applying the $X_G$ case, and then applying~$\omega$.
\end{proof}

\begin{remark}\label{r:not-flat}
Neither Corollary~\ref{c:constantclosed} nor its proof claims that the
homogenization of an orbit must necessarily be flat over the group;
that is, we do not require (or claim) that homogenizations of orbits
satisfy the flatness hypothesis in Theorem~\ref{t:equal-locus2}.  But
inverse multiplication brings an orbit into a position where flatness
does hold, and hence the closedness conclusion follows even if the
flatness hypothesis does not.
\end{remark}

\subsection{Finding big group actions}\label{sub:big-actions}

In this section, $T$ is an algebraic group acting on affine space via
a map $\beta: T \times \AA^n \to \AA^n$.  We are interested in finding
those elements of~$G$ which make a given subscheme $X \subseteq \AA^n$
invariant under~$T$, or invariant under a subgroup of~$T$ whose
dimension is as big as possible.  For terminology, we say that a group
acting on a space \emph{stabilizes} a subspace if the subspace is
preserved by the action (not necessarily pointwise); in contrast, we
say that the group \emph{fixes} a subspace if every point in the
subspace is fixed by the group action.

\begin{example}
Let $T = (\kk^*)^n$ be the $n$-dimensional torus acting diagonally
on~$\AA^n$.  If $X \subseteq \AA^n$ is a subvariety (reduced
irreducible subscheme) such that for some $\gamma \in G$ the
subvariety $\gamma.X \subseteq \AA^n$ is stabilized by a subtorus $T'
\subset T$ with $\dim T' = \dim X$, then by some definitions
$\gamma.X$ is already toric, and by others it becomes so after further
rescaling the variables \cite[Corollary~2.6]{binomialideals}.  In the
latter case, $\gamma'.X$ is toric for some other $\gamma' \in G$ if $G
\supseteq T$.  (Yet other definitions require $X$ to be normal; we do
not.)
\end{example}

We start with the following useful lemma.

\begin{lemma}\label{l:stabilizer}
The stabilizer of any ideal $I \subseteq \kk[\xx]$ is a closed
subgroup of~$T$.
\end{lemma}
\begin{proof}
For the scheme $X \subseteq \AA^n$ defined by~$I$, the containment
loci both for $X_T$ in~$\OO_T X$ and for $\OO_T X$ in~$X_T$ are closed
by Corollary~\ref{c:constantclosed}, so their coincidence locus is
closed.
\end{proof}

\begin{remark}\label{r:containment}
The two containment loci in the proof of Lemma~\ref{l:stabilizer} are
in fact equal, because $I.\tau \subseteq I \implies I.\tau = I$ for
$\tau \in T$.  Indeed, $I.\tau \subsetneq I$ implies $I.\tau^i
\subsetneq I.\tau^{i-1}$ for all $i \in \ZZ$ (including negative~$i$),
contradicting the noetherian property of~$\kk[\xx]$.
\end{remark}

If a subgroup $T' \subseteq T$ stabilizes a subscheme $X \subseteq
\AA^n$, then the action of~$T$ restricts to an action of~$T'$ on~$X$.
We now consider the question for which $\gamma \in G$ there is a large
subgroup of~$T$ stabilizing~$\gamma.X$.

\begin{prop}\label{p:largegroup}
Let $X \subseteq \AA^n$ be a closed subscheme.  Let further $G'
\subseteq G$ be the locus of those $\gamma \in G$ where the dimension
of the stabilizer subgroup $T(\gamma.X) \subseteq T$ is maximal.  Then
$G'$ is closed.
\end{prop}
\begin{proof}
Let $Y \subseteq T \times G$ be the coincidence locus for $\OO_T(\OO_G
X)$ and $(\OO_G X)_T$, viewed as schemes over $T \times G$.  The fiber
$Y_\gamma$ of $Y$ over $\gamma \in G$ is the stabilizer $T(\gamma.X)$
of~$X$ in~$T$.  The subset $Y \subseteq T \times G$ is closed by
Corollary \ref{c:constantclosed}, and $G'$ is the locus of points
$\gamma \in G$ such that $Y_\gamma$ has maximal dimension.

If $d(\gamma) = \dim Y_\gamma$ then $d(\gamma)$ is the local dimension
of $Y_\gamma$ at $\tau$ for all $\tau \in Y_\gamma$, in particular at
$\tau = 1_T$, because $Y_\gamma$ is a subgroup of $T$.  Upper
semicontinuity of fiber dimension locally on the source
\cite[13.1.3]{EGAIV} implies that the subset $Z \subseteq Y$ with
maximal local fiber dimension is closed in $Y$.  Therefore $1_T \times
G' = Z \cap (1_T \times G)$ is closed in the identity section $1_T
\times G$.  The result follows because $G \to 1_T \times G$
is~an~isomorphism.
\end{proof}

\begin{example}
Proposition~\ref{p:largegroup} says that stabilizer dimensions for
fibers of orbits $\OO_G X_G$ are upper semicontinuous.  In contrast,
semicontinuity can fail for families that are not orbits.  The
simplest instance has $T = \kk^*$, $S = \kk[s]$, and $I = \<sx,
x(x-1)\> \subset S[x]$.  Then $V(I)$ is the union of the $s$-axis with
the point $(0,1)$.  So for $s \neq 0$, $V(I)$ is stable under $T$,
while for $s = 0$ it is only stable under the trivial subgroup of $T$.
Thus, the locus where the stabilizer has maximal dimension is not
closed.
\end{example}

\begin{alg}\label{a:group}
Act to make stabilizing subgroup of maximal dimension
\end{alg}
\begin{alglist}
\begin{initlist}{Input}
\item[] two algebraic groups $T$ and $G$ acting on $\AA^n$
\item[] a closed subscheme $X \subset \AA^n$
\end{initlist}

\init{Output} the closed set $G' \subseteq G$ such that for any
	$\gamma \in G'$ the dimension of the subgroup of~$T$
	stabilizing $\gamma.X$ is maximal

\begin{routinelist}{compute}
\item[] $X_1 := \OO_T(\OO_G X_G)$
\item[] $X_2 := (\OO_G X_G)_T$
\item[] coincidence locus $Y \!\subseteq T \times G$ for $X_1$
	and~$X_2$ (Algorithm~\ref{a:equalfibers}; see also
	Remark~\ref{r:containment})
\item[] maximal fiber dimension locus $G' \subseteq G$ for $Y \to G$
	(Algorithm~\ref{a:dim})
\end{routinelist}

\routine{return} $G'$
\end{alglist}

\subsection{Finding multigradings}\label{sub:multigradings}

A \emph{$\ZZ^r$-multigrading} on $\kk[\xx]$ is specified by $n$
vectors $\aa_1, \dots, \aa_n \in \ZZ^r$ to serve as degrees of the
variables: $\deg x_i = \aa_i$.  The multigrading is \emph{faithful} if
$\aa_1, \dots, \aa_n$ generate~$\ZZ^r$.

\begin{remark}\label{r:A-graded}
Details on multigraded algebra in general can be found in
\cite[Chapter~8]{cca}.  A $\ZZ^r$-multigrading on $\kk[\xx]$
corresponds uniquely to the action on~$\kk^n$ of the
$r$-torus~$(\kk^*)^r$.  (References for this are hard to locate.  An
exposition appears in Appen\-dix~A.1 of the first arXiv version
of~\cite{grobGeom}, at
{\small\texttt{http://arxiv.org/abs/math/0110058v1}}.)  An abelian
group homomorphism $\ZZ^n \to \ZZ^r$ sending the standard basis to
$\aa_1, \dots, \aa_n$ corresponds (by applying the $\Hom_\ZZ(-,\kk^*)$
functor) to an algebraic group homomorphism $(\kk^*)^r \to (\kk^*)^n$
that is injective precisely when the multigrading is faithful.  The
assertion that $\deg x_i = \aa_i$ means that the $j^\th$ generator
$\tau_j$ of the $r$-torus acts on~$x_i$ by $x_i.\tau_j =
\tau_j^{a_{ij}} x_i$.  Geometrically speaking, multigraded (i.e.,
homogeneous) ideals correspond to subschemes of~$\kk^n$ that carry
$(\kk^*)^m$-actions.
\end{remark}

The following is the detailed algebraic phrasing of
Algorithm~\ref{a:group} when $T$ is the standard algebraic $n$-torus
acting diagonally on~$\kk^n$.

\begin{alg}\label{a:grading}
Act to find a faithful multigrading of maximal rank
\end{alg}
\begin{alglist}
\begin{initlist}{Input}
\item[] an ideal $I \subset \kk[\xx]$
\item[] a morphism $\omega^*:\kk[G]\otimes\kk[\xx] \to \kk[G]\otimes \kk[\xx]$
\end{initlist}

\begin{initlist}{Output}
\item[] an element $\gamma \in G$,
\item[] a nonnegative integer~$r$, and
\item[] vectors $\aa_1, \dots, \aa_n \in \ZZ^r$ defining a faithful
	multigrading that makes $I.\gamma$ homogeneous and $r$ as big
	as possible
\end{initlist}

\begin{routinelist}{define}
\item[] $R := \kk[t_1^\pm, \dots, t_n^\pm] = \kk[T]$, the Laurent
	polynomial ring
\item[] $\beta^*: \kk[\xx] \to R \otimes \kk[\xx]$, the algebra map
	specified by $x_i \mapsto x_i t_i$ for $1 \leq i \leq n$
\item[] $I_1 := (\id \otimes \beta^*)\big((\omega^*)^{-1}(I)\big) \subseteq R
	\otimes \kk[G]\otimes \kk[\xx]$
\item[] $I_2 := (\omega^*)^{-1}(I) R \subseteq R\otimes\kk[G]\otimes\kk[\xx]$
\end{routinelist}

\begin{routinelist}{compute}
\item[] coincidence locus $Y \subseteq \spec (\kk[G] \otimes R)$ of
	$I_1$ and $I_2$ (Algorithm~\ref{a:equalfibers})
\item[] maximal fiber dimension locus $G' \subseteq G$ for $Y \to G$
	(Algorithm~\ref{a:dim})
\item[] any point $\gamma \in G'$
\item[] generators for the kernel $J_\gamma$ of $R \to \kk[Y_\gamma]$
	(a binomial ideal \cite[Thm.~2.1]{binomialideals})
\item[] a basis for $\ZZ^n/L$, where $L = \<\aa - \bb \mid
	\mathbf{t}^\aa-\mathbf{t}^\bb\text{ is a generator of }J_\gamma\>$
\end{routinelist}

\begin{routinelist}{return}
\item[] group element $\gamma^{-1} \in G$,
\item[] integer $r := n - \rank(L)$, and
\item[] $n$ images of standard basis elements in $\ZZ^n/L$ expressed
	in the computed basis
\end{routinelist}
\end{alglist}\vspace{-.9ex}

\section{Detecting binomial and unital ideals}\label{s:binom}

This section presents algorithms to decide whether $I \subseteq
\kk[\xx]$ can be made binomial or unital by an automorphism
of~$\kk[\xx]$.  Let us recall the relevant definitions.

\begin{defn}\label{d:binomial}
A \emph{binomial} is a polynomial $\xx^\aa - \lambda\xx^\bb$ for some
$\lambda \in \kk$ and $\aa,\bb \in \NN^n$.  An ideal in $\kk[\xx]$ is
\begin{enumerate}
\item%
\emph{binomial} if it is generated by binomials, and
\item%
\emph{unital} (cf. \cite{mesoprimary}) if it is generated by monomials
and differences of monomials, meaning binomials $\xx^\aa -
\lambda\xx^\bb$ with $\lambda \in \{0,1\}$.
\item%
\emph{toric} if it is unital and prime.
\end{enumerate}
\end{defn} 

\subsection{Algorithm to locate binomial models}\label{sub:alg-binom}%

We start by recalling the construction of comprehensive Gr\"obner
bases \cite{weispfenning1992comprehensive, montes2010grobner} and
adapting it to our case.  Let $S$ be a domain which is a commutative
$\kk$-algebra and let $I \subseteq S[\xx]$ be an ideal.  Fix any term
order.  Buchberger's algorithm for $I$ over the generic point---that
is, in the polynomial ring $\kappa_0[\xx]$ over the fraction field
$\kappa_0$ of~$S$---finds a reduced Gr\"obner basis.  It finishes in a
finite number of steps.  In step~$i$ of the algorithm, one needs to
assume that a leading coefficient $f_i\in S$ is nonzero.  The output
of the algorithm is a finite set of functions $g_j\in I$.  Consider
the ideal $J = \big\<\prod_i f_i\big\> \subseteq S$.  We claim that
over any
point $\pp \in \spec S \minus V(J)$, the reduction of $g_j$ is the
reduced Gr\"obner basis for the reduction of $I$.  Indeed, over such a
point the usual Buchberger algorithm makes exactly the same steps as
the algorithm run over the generic point.  Repeating the procedure for
each irreducible component of $V(J)$ yields:
\begin{itemize}
\item%
a partition of $\spec S$ into irreducible, relatively open sets~$U_i$;
and
\item%
for each $U_i$ a finite set of polynomials $g_{ij}$ in~$I$ that
specializes over any point of~$U_i$ to a reduced Gr\"obner basis of
the reduction of $I$.
\end{itemize}
We call this data structure a \emph{relative reduced Gr\"obner basis}.

The following algorithm forces all coefficients except for the leading
coefficient and (at most) one other to be~$0$; more precisely, it
computes the locus where this is possible.  The result is the binomial
locus because an ideal is binomial if and only if some (equivalently,
every) reduced Gr\"obner basis consists of binomials
\cite[Corollary~1.2]{binomialideals}.

\begin{alg}\label{a:binomial}
Find the locus of binomial fibers
\end{alg}
\begin{alglist}
\begin{initlist}{Input}
\item[] commutative $\kk$-algebra $S$ that is an integral domain
\item[] ideal $I \subseteq S[\xx]$
\end{initlist}

\init{Output} constructible subset $Z \subseteq \spec S$ such that
	$I\kappa_\pp[\xx]$ is binomial precisely for $\pp \in Z$

\routine{compute} relative reduced Gr\"obner basis $\{(U_i,(g_{ij})_j)
	\with i = 1,\dots,k\}$ of~$I$ over~$S$

\routine{initialize} $i = 1$

\routine{while} $i = 1,\dots,k$ \procedure{do}

\begin{routinelist}{define}
\item[] $\NL_i := \prod_j \{$nonleading monomials of~$g_{ij}\}$, which
	is a cartesian product of sets
\item[] ideal $\num F := \<$numerator of~$f \with f \in F\> \subseteq
	S$ for any set $F \subseteq \kappa_0$ of fractions
\item[] ideal $J_{i\ell} \!:=\!	\sum_j\num\{$coefficients of~$g_{ij}$
	neither $\ell_j^{\,\,\th}$ nor leading$\} \subseteq S$ for
	$\ell \in \NL_i$
\end{routinelist}

\routine{advance} $i \from i+1$

\routine{end}{\procedure{while-do}}

\routine{return} $\bigcup_{i=1}^k Z_i$ for $Z_i = U_i \cap
	\big(\bigcup_{\ell \in \NL_i} V(J_{i\ell})\big)$
\end{alglist}

\subsection{Algorithm to locate unital models}\label{sub:alg-unital}

It is possible for a group~$G$ to have the power to transform a given
ideal~$I$ into binomial form without $G$ being able to transform~$I$
into unital form, even though a larger group $G$ could succeed in
making~$I$ unital.  Trivial examples exist: any principal non-unital
binomial ideal with $G = \{1\}$ suffices.  But it is even possible for
$G$ to contain the entire torus.  (See also
Example~\ref{e:non-unital}.)

\begin{example}\label{e:binomial-not-unital}
Consider the binomial ideal $I = \<xz, z(z-1), z(y-2), x(x-1),
x(y-1)\>$.  If $z \neq 0$ then the only point in $V(I)$ is $p_1 =
(0,2,1)$, and if $x \neq 0$ then the only point is $p_2 = (1,1,0)$.
The map that fixes $x$ and~$y$ but sends $y \mapsto y+z$ makes $I$
unital.
	
On the other hand, if $\tau \in (\CC^*)^3$ then $y(\tau.p_1) \neq 1$
or $y(\tau.p_2) \neq 1$.  As both cases are analogous, say
$y(\tau.p_1) \neq 1$.  Any monomial evaluated on $(0,1,1)$ is either
$0$ or $1$. Further it is equal to $0$ if and only if it is equal to
$0$ when evaluated on $\tau.p_1$.  Thus if two monomials are equal
when evaluated on $\tau.p_1$ then they are also equal when evaluated
on $(0,1,1)$.  Hence, any unital binomial or monomial that vanishes on
$\tau.p_1$ vanishes also on $(0,1,1)$.  If $I.\tau$ were unital, this
would contradict the fact that $\tau.p_1$ is the only point with $z
\neq 0$.
\end{example}

The main idea of our algorithm is that unital ideals $I$ can be
characterized by the fact that the variety $V(I)$ is closed under
coordinatewise multiplication.  It turns out to be more convenient to
work with the algebraic counterpart of the multiplication, which is
the diagonal map $\Delta$.

Precisely, let $\Delta: \kk[\xx] \to \kk[\xx] \otimes_\kk \kk[\xx]$ be
the algebra homomorphism defined by $\Delta(x_i) = x_i \otimes x_i$.
This makes $\kk[\xx]$ into a bialgebra.  The induced map $\Delta^*:
\AA^n \times \AA^n \to \AA^n$ is the coordinatewise multiplication
map.

\begin{prop}\label{p:monoidal}
The following are equivalent for an ideal $I \subseteq \kk[\xx]$.
\begin{enumerate}
\item\label{eq:m1:1}%
$I$ is unital.
\item\label{eq:m1:2}%
$I$ is a coideal with respect to $\Delta$; that is, $\Delta(I)
\subseteq \kk[\xx] \otimes_\kk I + I \otimes_\kk \kk[\xx]$.
\end{enumerate}
\end{prop}
\begin{proof}
The implication ``{\ref{eq:m1:1} $\implies$ \ref{eq:m1:2}}''
follows from the definitions, because
\begin{align*}
  \Delta(\xx^\aa) &=\xx^\aa\otimes \xx^\aa\in \kk[\xx] \otimes_\kk I + I
  \otimes_\kk \kk[\xx],
\\
  \Delta(\xx^\aa - \xx^\bb) &=\xx^\aa \otimes (\xx^\aa - \xx^\bb) +
  (\xx^\aa - \xx^\bb) \otimes \xx^\bb \in \kk[\xx] \otimes_\kk I + I
  \otimes_\kk \kk[\xx].
\end{align*}
The implication ``{\ref{eq:m1:2} $\implies$ \ref{eq:m1:1}}'' is
essentially due to Artin,
cf.~\cite[Remark~p.\thinspace{}15]{binomialideals}.  Let us recall the
argument.

If $I$ is a coideal, then $\Delta$ induces $\ol\Delta: \kk[\xx]/I \to
\kk[\xx]/I \otimes_k \kk[\xx]/I$ compatible with the canonical
projection $\pi: \kk[\xx] \to \kk[\xx]/I$.  Now suppose that $\pi(m_0)
= \sum_{i=1}^\ell \lambda_i \pi(m_i)$ is a relation among classes of
monomials with $\ell$ minimal (and hence $\lambda_i \neq 0$ for $i >
0$).~~Then
\begin{align*}%
\sum_{i>0}\lambda_i\pi(m_i)\otimes \pi(m_i)
  &= \ol\Delta\Big(\sum_{i>0}\lambda_i\pi(m_i)\Big)
\\&= \ol\Delta\big(\pi(m_0)\big) = \pi(m_0) \otimes \pi(m_0)
\\&= \Big(\sum_{i>0} \lambda_i \pi(m_i)\Big) \otimes
     \Big(\sum_{i>0} \lambda_i \pi(m_i)\Big)
\\&= \sum_{i,j>0} \lambda_i \lambda_j \pi(m_i)\otimes \pi(m_j).
\end{align*}
As $\ell$ is minimal, the $\pi(m_i)$ are independent for $i > 0$.
Hence the $\pi(m_i) \otimes \pi(m_j)$ are also independent.  But
$\lambda_i\lambda_j \neq 0$ for $i \neq j$, so $\ell = 1$ and
$\lambda_1 = \lambda_1^2$, whence $\lambda_1 = 1$.
\end{proof}

\begin{alg}\label{a:checkUnital}
Find the locus of unital fibers
\end{alg}
\begin{alglist}
\init{Input} ideal $I \subseteq S[\xx]$

\init{Output} constructible subset $Z \subseteq \spec S$ such that
	$I\kappa_\pp[\xx]$ is unital precisely for $\pp \in Z$

\begin{routinelist}{define}
\item[] $I_1 := \Delta(I) \subseteq S[\xx] \otimes_S S[\xx]$
\item[] $f_1: S[\xx]\to S[\xx] \otimes S[\xx]$ via $s \mapsto s\otimes 1$
\item[] $f_2: S[\xx]\to S[\xx] \otimes S[\xx]$ via $s \mapsto 1\otimes s$
\item[] $I_2 := f_1(I) + f_2(I)\subseteq S[\xx] \otimes S[\xx]$
\end{routinelist}

\routine{return} containment locus of $I_1$ in~$I_2$ as families
	over~$S$ (Algorithm~\ref{a:equalfibers})

\end{alglist}

\subsection{Unital loci from group actions}\label{sub:structural}

This subsection assumes $\kk$ is uncountable.

Algorithm~\ref{a:checkUnital} works over an arbitrary base.  The
Section~\ref{sec:T-act} setup, where the base is a group of ambient
automorphisms, yields a further structural property of unital~loci.

\begin{prop}\label{p:finiteunion}
The set of points $\gamma \in G$ such that $I.\gamma$ is unital is a
finite union\vspace{-.1ex}
$$%
  \bigcup_{\gamma \in U} G(I)\gamma
$$
of cosets of the stabilizer~$G(I)$.  In particular, this set is
closed.
\end{prop}
\begin{proof}
First, $\kk[\xx]$ has only countably many monomials.  As every unital
ideal is generated by finitely many monomials and differences of
monomials, there are also only countably many unital ideals of the
form~$I.\gamma$.  The set of points $\gamma \in G$ such that
$I.\gamma$ is unital has the form
$$%
  \bigcup_{\gamma \in U} G(I)\gamma
$$
and is constructible by Theorem~\ref{t:equal-locus2} and
Proposition~\ref{p:monoidal}.  Lemma~\ref{l:construct} implies that
the union has to be finite.  Finally, stabilizers are closed by
Lemma~\ref{l:stabilizer}, so their cosets are, as well, and so are
finite unions thereof.
\end{proof}

\begin{example}\label{e:implies}
Under the componentwise action of the usual torus, any binomial prime
ideal becomes unital---and hence toric---after rescaling the variables
appropriately \cite[Corollary~2.6]{binomialideals}.  The prime
assumption here is essential, as Example~\ref{e:non-unital}~shows.
\end{example}

\begin{example}\label{e:non-unital}
It need not be possible to find a group action---or any family with
isomorphic fibers---taking a given ideal to a unital one, even if the
original is a binomial ideal.  Indeed, the ideal $I = \<u^5(u - v),
v^5(u - 2v)\>$ is binomial but $\CC[u,v]/I$ is not (abstractly)
isomorphic to a quotient of a polynomial ring $R$ by a unital
ideal~$J$.

To see why, suppose such a $J$ exists and consider $R$ of smallest
possible dimension.  Since $\CC[u,v]/I$ is supported at the origin
(both $u^{11}$ and~$v^{11}$ lie in~$I$), the quotient $R/J$ is
supported at a single point.  As $J$ is unital, the variables can be
numbered so that its support point is $\xi = (1,\dots, 1,0,\dots, 0)$.
Extending $J$ by all $x_i$ such that $x_i(\xi) = 0$ then yields a
reduced ideal \cite[Theorem~9.12]{mesoprimary}.  Hence $x_1(\xi) = 1
\implies x_1 - 1 \in J$, which contradicts minimality of~$\dim R$.
Thus $\xi = 0$.  Note that $R/J$ has tangent space of dimension~$2$.
Consequently, $R = \CC[x_1,x_2]$.  Indeed, if $\dim R > 2$ then $J$
contains a binomial of the form $x - g$, where $g$ is a monomial and
$x$ is a variable.  If $x \nmid g$ then $x - g$ eliminates~$x$,
contradicting minimality of~$\dim R$.  And if $x \mid g$ then
repeatedly replace $x$ in~$g$ to get binomials of the form $x - g_i
\in J$ with $\deg g_i \to \infty$; the fact that $R/J$ is Artinian
proves that $x \in J$, again leading to a contradiction.

Since $I$ is a complete intersection, $J = \<f_1,f_2\>$.  Since $R/J$
is supported at the origin, it equals its localization at
$\<x_1,x_2\>$.  As $J$ is unital, Nakayama's lemma produces unital
binomials $b_1, b_2$ such that $J = \<b_1,b_2\>$ as ideals of the
localization~$\tilde R$.  More generally, minimal systems of
generators of an ideal in a local ring have the same cardinality.

The local Hilbert function of $\CC[u,v]/I$ implies that (i)~neither
$b_1$ nor~$b_2$ has nonzero monomials of degree less than~$6$, and
(ii)~each of them contains a monomial of degree~$6$.

Now note that there are, up to scaling, precisely two distinct pairs
of nonzero elements $\ell_{1,i},\ell_{2,i} \in \<u,v\>/\<u,v\>^2$ for
$i = 1,2$ such that $\ell_{1,i}^a \ell_{2,i}^b = 0 \in
\<u,v\>^6/\<u,v\>^7 \subset \CC[u,v]/(I+\<u,v\>^7)$ for $a + b = 6$.
Further, the exponents satisfy $a = 1$ and $b = 5$ or $a = 5$ and $b =
1$.  Indeed, suppose $\ell_{1,i}, \ell_{2,i}$ satisfy the above
condition.  If $\ell_{1,i} \neq u$ and $\ell_{2,i} \neq u-v$ (which
provides one possible pair) it must be that
$$%
  \ell_{1,i}^a \ell_{2,i}^b=\lambda u^5(u-v)+v^5(u-2v)
$$
for some $\lambda \in \CC$.  Dividing by $u^6$ and setting $t = v/u$,
the equality above means that the polynomial $P(t) = \lambda(1 - t) +
t^5 - 2t^6$ has two roots of multiplicity $a$ and~$b$, respectively.
Swapping if necessary, assume $a \geq 3$.  Vanishing of the second
derivative forces the root to satisfy $20t^3 - 60t^4 = 0$, i.e.~$t =
0$ or $t = 1/3$.  If $t = 0$ is the root of~$P(t)$, then $\lambda = 0$
and a second pair of linear forms $\{v, u - 2v\}$ arises.  When $t =
1/3$, nonvanishing of $(d^4P(t)/dt^4)(1/3)$ means it can be a root of
multiplicity at most three, and a second root of multiplicity three
would be required, but none exists.

The next goal is to exclude cases where $b_1, b_2$ are monomials or
fail to be homogeneous.  So assume $b_1$ is a monomial or
inhomogeneous binomial.  Then $b_1$ has image $x_1^ax_2^{6-a}$ in
$\tilde R/\<x_1,x_2\>^7$; without loss of generality assume it
is~$x_1^5x_2$, providing the first pair $\{\ell_{1,1},
\ell_{2,1}\}=\{x_1,x_2\}$.  This forces $b_2$ to be homogeneous, as it
would otherwise provide a second monomial, necessarily equal
to~$x_1x_2^5$, contradicting the fact that the pairs
$\ell_{1,i},\ell_{2,i}$ are distinct for $i = 1,2$.  Hence, $b_2$ is a
homogeneous degree $6$ binomial.

Observe that for any nonzero element $h \in \<u,v\>/\<u,v\>^2
\subseteq \CC[u,v]/I + \<u,v\>^2$ there is at most a $1$-dimensional
family of elements of $\<u,v\>^5/\<u,v\>^6$ that annihilate it modulo
$I + \<u,v\>^7$.  Indeed, otherwise $h P_1 = \alpha_1 u^5(u-v) +
\beta_1 v^5(u - 2v)$ and $h P_2 = \alpha_2 u^5(u - v) + \beta_2 v^5(u
- 2v)$ for some linearly independent polynomials $P_1, P_2$ of
degree~$5$.  However, then $h$ would be a common divisor of $u^5(u -
v)$ and $v^5(u - 2v)$, which is not possible.

Hence $b_2$ cannot be divisible by $x_1$ or~$x_2$ and thus must equal
$x_1^6 - x_2^6$.  But then there is no second pair of linear forms
$\ell_{1,2}, \ell_{2,2} \in \CC[x_1,x_2]/(J+\<x_1,x_2\>^7)$ such that
$\ell_{1,2}^5 \ell_{2,2}^{\phantom 1} = 0$, because a polynomial of
the form $\lambda t + t^6 - 1$ cannot have a root of multiplicity~$3$,
so it must have more than two distinct roots.  This concludes the
proof that $b_1$ and~$b_2$ are both homogeneous and that neither is a
monomial.  But this means that $x_1 - x_2$ divides both $b_1$
and~$b_2$, which contradicts $\dim \tilde R/J = 0$.
\end{example}

%
%

\section{Detecting toric ideals and varieties}\label{s:toric}

This section presents an algorithm to check whether a given
homogeneous prime ideal defines a variety that is abstractly
isomorphic to a toric one (Section~\ref{sub:toric}).  While this could
be done using our earlier algorithms, the hypothesis that $I$ is prime
allows significant simplifications (Section~\ref{sub:faster}).

\subsection{Faster algorithm for fiber containment of an irreducible family}
\label{sub:faster}

The procedure in Section~\ref{sub:toric} is made faster by the
following alternative to Algorithm~\ref{a:equalfibers} in the special
case that the ideal that is requested to be smaller (that is,~$I_1$)
has fibers that are known to be prime.  The advantage is that we
expect this algorithm to run much faster than
Algorithm~\ref{a:equalfibers}.

\begin{alg}\label{a:equalfibersIrreduc}
Compute containment locus families when one has prime fibers
\end{alg}
\begin{alglist}
\begin{initlist}{Input}
\item[] ideal $I_1 \subseteq S[\xx]$ with prime fiber $I_1
	\kappa_\pp[\xx]$ for every prime ideal $\pp \in \spec S$
\item[] ideal $I_2 \subseteq S[\xx]$
\end{initlist}

\init{Output} containment locus for $I_1$ in~$I_2$ as a constructible
	set
\pagebreak
\begin{routinelist}{define}
\item[] $X := V(I_1 + I_2)$
\item[] $B := \spec S$
\item[] $d := \dim V(I_1)$
\end{routinelist}

\routine{return} locus $B^{\geq d} \subseteq B$ where fibers have
	dimension $\geq d$ (Algorithm~\ref{a:dim} or \cite{kemper})
\end{alglist}

\begin{prop}\label{p:correct-prime-alg}
Algorithm~\ref{a:equalfibersIrreduc} is correct.
\end{prop}
\begin{proof}
In any finitely generated commutative $\kk$-algebra $R$ (we apply it
to $\kappa_\pp[\xx]$), if one ideal~$I$ contains a prime ideal~$J$,
then $I = J$ if and only if $\dim(R/I) = \dim(R/J)$.
\end{proof}

\begin{example}\label{e:macaulay2}
This is a Macaulay2 \cite{macaulay2} demonstration of testing whether
there exists $\gamma \in G = (\CC,+)$ such that $I.\gamma$ is toric,
where $I = \<xy + 2y^2 - 1\>$ and $(\CC,+)$ acts on~$\CC^2$ by
$\gamma.(x,y) = (x + \gamma y, y)$.  The code applies
Algorithm~\ref{a:checkUnital}, relying on the fast
Algorithm~\ref{a:equalfibersIrreduc} instead of
Algorithm~\ref{a:equalfibers} to compute containment.  Moreover, in
view of Remark~\ref{r:inverse} it works with the inverse group action,
so as to obtain the ideal~$I$ as the image of a map.

\begin{verbatim}
R = QQ[x,y];
I = ideal (x*y+2*y^2-1);
GxR = QQ[a,x,y];
GxRxR = QQ[aa,xx1,yy1,xx2,yy2];
alpha = map(GxR,R,{x-a*y, y}); -- action by the inverse
IdotDelta = map(GxRxR,GxR,{aa,xx1*xx2,yy1*yy2});
I1 = IdotDelta(alpha(I));
idf1 = map(GxRxR, GxR,{aa,xx1,yy1});
idf2 = map(GxRxR, GxR,{aa,xx2,yy2});
I2 = idf1(alpha(I))+idf2(alpha(I));
I3 = I1+I2;
BPrimn = ideal(sub(0,GxRxR)), BPrimo = ideal(sub(1,GxRxR));
while (not ((radical BPrimn)==BPrimo))
   do {
       L = 0;
       for i from 1 to (2*(dim I))
       do{
          L = L+ideal(sub(random(1,GxRxR),{aa=>random QQ}))
         };
       I4 = I3+L;
       BPrimo = BPrimn,
       BPrimn = radical(BPrimo+eliminate(I4,{xx1,yy1,xx2,yy2}));
      }
M = map(GxR,GxRxR,{a,x,y,0,0});
eliminate(M(BPrimn)+alpha(I),a);
eliminate(M(BPrimn)+alpha(I),{x,y})
\end{verbatim}\vspace{1.5ex}
\end{example}

\subsection{Toric varieties}\label{sub:toric}

This section shows how to detect whether a projective variety is toric
without any prespecified group or other family of ambient
automorphisms.

\begin{alg}\label{a:checkProjToric}
Decide whether a normal projective variety is abstractly toric
\end{alg}
\begin{alglist}
\init{Input} normal projective variety $X \subseteq \PP^n$

\init{Output} a projective toric embedding of~$X$ if it is toric,
	else \procedure{false}

\begin{routinelist}{compute}
\item[] a projectively normal Veronese embedding of~$X\!$
	\cite[Exercise~II.5.14]{hartshorne}
\item[] a homogeneous prime ideal $I \subseteq S$ such that $X =
	\Proj(S/I)$
\item[] $\gamma \in \GL_N$ such that $I.\gamma^{-1}$ is toric
	(Algorithm~\ref{a:checkUnital})
\end{routinelist}

\routine{return} $\gamma.X$ or \procedure{false}, accordingly
\end{alglist}

\begin{thm}\label{t:checkProjToric}
Algorithm~\ref{a:checkProjToric} is correct.
\end{thm}
\begin{proof}
The re-embedding can be done by attempting successively higher
Veronese maps and checking whether each is projectively normal.  The
cited source guarantees that this procedure terminates.

It remains only to show that if $X$ is toric, then there really exists
$\gamma \in \GL_N$ such that $\gamma.X$ is equivariantly embedded.
The embedding of $X$ distinguishes a very ample divisor $L$ on $X$.
If $X$ is toric, then $L$ is equivalent to a toric divisor~$L'$ by
\cite[Theorem~4.2.1]{CLS-toricVarieties}.  Therefore, the projectively
normal embedding yields a surjection $\Gamma(\PP^N,\mathcal{O}(1)) \to
\Gamma(X, L)$.  In particular, $X$ is toric if and only if there
exists an automorphism of $\PP^N$ under which $I(X)$ goes to a toric
ideal.  But all automorphisms of $\PP^N$ are (projectively) linear, so
the desired one is represented by some matrix $\gamma \in \GL_N$.
\end{proof}

\begin{remark}\label{r:toric}
To check if a variety $X$ is toric it is essential that $X$ be
projective.  For example, it is an open problem to decide whether the
affine variety defined by the ideal $\<x + x^2y + z^2 + t^3\> \subset
\CC[x,y,z,t,w]$ is isomorphic to~$\AA^4$
\cite[Remark~5.3]{mm-topics-toric}.
\end{remark}

\section{Conclusion}

In retrospect, many of our algorithms apply not only to a group of
automorphisms of an affine space but to an arbitrary family of
transformations.  To be precise, fix an arbitrary morphism $\alpha: Y
\times \AA^n \to \AA^n$, thought of as a family of maps $\AA^n \to
\AA^n$ parametrized by~$Y$.  For a $\kk$-valued closed point $\eta \in
Y$ denote by $\alpha_\eta: \AA^n \to \AA^n$ the morphism obtained by
composition of the isomorphism $\AA^n \to \{\eta\} \times \AA^n$ and
(the restriction of) $\alpha$.  Given an affine variety $X \subseteq
\AA^n$ one may ask for the locus of points $\eta \in Y$ such that
$\alpha^{-1}_\eta(X)$ is defined by a unital ideal.  In the group
action setting, where $Y = G$ is a group and~$\alpha$ is an action
(Section~\ref{sub:actions}), working with images and preimages are
more or less equivalent: they amount to taking orbits for $\gamma$
or~$\gamma^{-1}$.  But in this more general setting, working with
preimages means computing inverse images of subschemes (images of
ideals), which is trivial, instead of computing images of schemes
(kernels of ring maps), which is a hard problem known as
implicitization.  Furthermore, the inverse image is closed, whereas
for images of morphisms the global closure may be not compatible with
closure fiberwise, which creates additional problems.

\begin{remark}\label{r:extend}
Even using preimages of subschemes instead of images, extending our
algorithms to this more general setting requires special attention.
For example, although two families over~$Y$ can still be compared as
in Algorithm~\ref{a:checkUnital}, the dimension argument in
Section~\ref{sub:faster} no longer necessarily applies.
\end{remark}

\begin{remark}\label{r:binomial}
In contrast, the methods to test binomiality in
Section~\ref{sub:alg-binom} adapt verbatim to the case of arbitrary
maps, as they only rely on comprehensive Gr\"obner bases.
\end{remark}

\begin{remark}\label{r:monomial}
It is similarly easy to check if an ideal is generated by monomials in
a similar way to Algorithm~\ref{a:binomial}.  Indeed, for each $U_i$
one only needs to check if it is possible for all coefficients of
nonleading monomials to vanish.  Alternatively, note that an ideal $I
\subseteq \kk[x_1,\dots, x_n]$ is monomial if and only if it is stable
under the whole torus~$T = (\kk^*)^n$ and apply
Algorithm~\ref{a:group}.
\end{remark}

In view of our results, we find the following three problems of
particular importance.

\begin{prob}
Is the problem of determining if an affine variety is affine space
decidable? Equivalently, is it decidable to test if a finitely
generated $\kk$-algebra is a polynomial ring?
\end{prob}

\begin{prob}
Is the problem of determining if a projective (nonnormal) variety
admits a torus action with a dense orbit decidable?
\end{prob}

\begin{prob}
Is the problem of determining if a given affine variety is toric
decidable?
\end{prob}

The last problem may be asked both for normal and arbitrary affine
varieties.



\begin{thebibliography}{DMM10}
\raggedbottom

\bibitem[BM16]{brunat-montes2016}
J. M. Brunat and A. Montes.
  \emph{Computing the canonical representation of constructible sets},
  Math. Comput. Sci. 10.1 (2016), pp. 165--178.
  doi: \href{http://dx.doi.org/10.1007/s11786-016-0248-2}
  {10.1007/s11786-016-0248-2}

\bibitem[BW07]{buczynska2007geometry}
W. Buczynska and J. A. Wisniewski.
  \emph{On the geometry of binary symmetric models of phylogenetic trees},
  J. Eur. Math. Soc. (JEMS) 9.3 (2007), pp. 609--635.
  doi: \href{http://dx.doi.org/10.4171/JEMS/90}{10.4171/JEMS/90}

\bibitem[CFS11]{casanellas2011relevant}
M. Casanellas and J. Fern\'andez-S\'anchez.
  \emph{Relevant phylogenetic invariants of evolutionary models},
  J. Math. Pures Appl. 96.3 (2011), pp. 207--229.
  doi: \href{http://dx.doi.org/10.1016/j.matpur.2010.11.002}
  {10.1016/j.matpur.2010.11.002}

\bibitem[CLS11]{CLS-toricVarieties}
D. A. Cox, J. B. Little, and H. K. Schenck.
  \emph{Toric varieties}.
  American Mathematical Soc., 2011.

\bibitem[DK09]{draisma2009ideals}
J. Draisma and J. Kuttler.
  \emph{On the ideals of equivariant tree models},
  Math. Ann. 344.3 (2009), pp. 619--644.
  doi: \href{http://dx.doi.org/10.1007/s00208-008-0320-6}
  {10.1007/s00208-008-0320-6}

\bibitem[DMM10]{primDecomp}
A. Dickenstein, L. F. Matusevich, and E. Miller.
  \emph{Combinatorics of binomial primary decomposition},
  Math. Z. 264.4 (2010), pp. 745--763.
  doi: \href{http://dx.doi.org/10.1007/s00209-009-0487-x}
  {10.1007/s00209-009-0487-x}

\bibitem[Eis95]{Eis}
D. Eisenbud.
  \emph{Commutative algebra, with a view toward algebraic geometry}.
  Vol. 150. Graduate Texts in Mathematics. Springer-Verlag, New York, 1995.
  doi: \href{http://dx.doi.org/10.1007/978-1-4612-5350-1}{10.1007/978-1-}
  \href{http://dx.doi.org/10.1007/978-1-4612-5350-1}{4612-5350-1}

\bibitem[Eri$^+$05]{Phylogeneticalgebraicgeometry}
N. Eriksson, K. Ranestad, B. Sturmfels and S. Sullivant.
  \emph{Phylogenetic algebraic geometry},
   in: \emph{Projective varieties with unexpected properties},
   Walter de Gruyter GmbH \& Co. KG, Berlin, 2005, pp. 237--255.

\bibitem[ES96]{binomialideals}
  D. Eisenbud and B. Sturmfels.
  \emph{Binomial ideals},
  Duke Math. J. 84.1 (1996), pp. 1--45.
  doi: \href{http://dx.doi.org/10.1215/S0012-7094-96-08401-X}
  {10.1215/S0012-7094-96-08401-X}

\bibitem[Ess00]{polyAut}
A. van den Essen.
  \emph{Polynomial Automorphisms}.
  Birkh\"auser Basel, 2000.
  doi: \href{http://dx.doi.org/10.1007/978-3-0348-8440-2}{10.1007/}
  \href{http://dx.doi.org/10.1007/978-3-0348-8440-2}{978-3-0348-8440-2}

\bibitem[Gro66]{EGAIV}
  A. Grothendieck.
  \emph{\'El\'ements de g\'eom\'etrie alg\'ebrique. IV. \'Etude locale des
  sch\'emas et des morphismes de sch\'emas. III},
  Inst. Hautes \'Etudes Sci. Publ. Math. 28 (1966), pp. 5--248.
  doi: \href{http://dx.doi.org/10.1007/BF02684343}
  {10.1007/BF02684343}

\bibitem[GS]{macaulay2}
D. R. Grayson and M. E. Stillman.
  \emph{Macaulay2, a software system for research in algebraic geometry}.
  \url{http://www.math.uiuc.edu/Macaulay2/}

\bibitem[GSS05]{garcia2005algebraic}
  L. D. Garcia, M. Stillman, and B. Sturmfels.
  \emph{Algebraic geometry of Bayesian networks},
  J. Symb. Comput. 39.3 (2005), pp. 331--355.
  doi: \href{http://dx.doi.org/10.1016/j.jsc.2004.11.007}
  {10.1016/j.jsc.2004.11.007}

\bibitem[Har77]{hartshorne}
R. Hartshorne. \emph{Algebraic geometry}.
  Graduate Texts in Mathematics, No. 52. Springer-Verlag, New
  York-Heidelberg, 1977.
  doi: \href{http://dx.doi.org/10.1007/978-1-4757-3849-0}
  {10.1007/978-1-4757-3849-0}

\bibitem[HP89]{hendy1989framework}
  M. D. Hendy and D. Penny.
  \emph{A framework for the quantitative study of evolutionary trees},
  Systematic Biology 38.4 (1989), pp. 297--309.
  doi: \href{http://dx.doi.org/10.2307/2992396}
  {10.2307/2992396}

\bibitem[Kem07]{kemper}
G. Kemper.
  \emph{Morphisms and constructible sets: Making two theorems of Chevalley
  constructive}, Preprint (2007). \
  \href{https://www-m11.ma.tum.de/fileadmin/w00bnb/www/people/kemper/kemper.chevalley.pdf}{\tt https://www-m11.ma.tum.de/fileadmin/w00bnb/www/people/}
  \href{https://www-m11.ma.tum.de/fileadmin/w00bnb/www/people/kemper/kemper.chevalley.pdf}{\tt kemper/kemper.chevalley.pdf}

\bibitem[KM05]{grobGeom}
  A. Knutson and E. Miller.
  \emph{Gr\"obner geometry of Schubert polynomials},
  Ann. of Math. (2) 161.3 (2005), pp. 1245--1318.
  doi: \href{http://dx.doi.org/10.4007/annals.2005.161.1245}
  {10.4007/annals.2005.161.1245}

\bibitem[KM14]{mesoprimary}
  T. Kahle and E. Miller.
  \emph{Decompositions of commutative monoid congruences and binomial ideals},
  Algebra Number Theory 8.6 (2014), pp. 1297--1364.
  doi: \href{http://dx.doi.org/10.2140/ant.2014.8.1297}
  {10.2140/ant.2014.8.1297}

\bibitem[KMO16]{soccular}
  T. Kahle, E. Miller, and C. O'Neill.
  \emph{Irreducible decomposition of binomial ideals},
  Compos. Math. 152.6 (2016), pp. 1319--1332.
  doi: \href{http://dx.doi.org/10.1112/S0010437X16007272}
  {10.1112/S0010437X16007272}

\bibitem[Mic17]{mm-topics-toric}
  M. Micha\l{}ek.
  \emph{Selected topics on Toric Varieties},
  in: The 50th Anniversary of Gr\"obner Bases. Vol. 75.
  Advanced Studies in Pure Mathematics. 2017.

\bibitem[MOZ15]{MOZ}
  M. Micha\l{}ek, L. Oeding, and P. Zwiernik.
  \emph{Secant cumulants and toric geometry},
  Int. Math. Res. Not. 12 (2015), pp. 4019--4063.
  doi: \href{http://dx.doi.org/10.1093/imrn/rnu056}
  {10.1093/imrn/rnu056}

\bibitem[MPS16]{MPS}
M. Micha\l{}ek, A. Perepechko, and H. S\"u\ss.
  \emph{Flexible affine cones and flexible coverings},
  Preprint (2016).
  arXiv: \href{http://arxiv.org/abs/1612.01144}{math.AG/1612.01144}

\bibitem[MV17]{michalek2016finite}
  M. Micha\l{}ek and E. Ventura.
  \emph{Finite phylogenetic complexity and combinatorics of tables},
  Algebra Number Theory 11.1 (2017), pp. 235--252.
  doi: \href{http://dx.doi.org/10.2140/ant.2017.11.235}
  {10.2140/ant.2017.11.235}

\bibitem[MS05]{cca}
E. Miller and B. Sturmfels.
  \emph{Combinatorial commutative algebra}.
  Vol. 227. Graduate Texts in Mathematics. Springer-Verlag, New York, 2005.
  doi: \href{http://dx.doi.org/10.1007/b138602}
  {10.1007/b138602}

\bibitem[MW10]{montes2010grobner}
  A. Montes and M. Wibmer.
  \emph{Gr\"obner bases for polynomial systems with parameters},
  J. Symb. Comput. 45.12 (2010), pp. 1391 --1425.
  doi: \href{http://dx.doi.org/10.1016/j.jsc.2010.06.017}
  {10.1016/j.jsc.2010.06.017}

\bibitem[Rai12]{raicu2012secant}
  C. Raicu.
  \emph{Secant varieties of Segre-Veronese varieties},
  Algebra Number Theory 6.8 (2012), pp. 1817--1868.
  doi: \href{http://dx.doi.org/10.2140/ant.2012.6.1817}
  {10.2140/ant.2012.6.1817}

\bibitem[SS05]{sturmfels2005toric}
  B. Sturmfels and S. Sullivant.
  \emph{Toric ideals of phylogenetic invariants},
  Journal of Computational Biology 12.2 (2005), pp. 204--228.
  doi: \href{http://dx.doi.org/10.1089/cmb.2005.12.204}
  {10.1089/cmb.2005.12.204}

\bibitem[SU04]{SU04}
  I. P. Shestakov and U. U. Umirbaev.
  \emph{The tame and the wild automorphisms of polynomial rings in three variables},
  J. Am. Math. Soc. (JAMS) 17.1 (2004), pp. 197--227.
  doi: \href{http://dx.doi.org/10.1090/s0894-0347-03-00440-5}
  {10.1090/s0894-0347-03-00440-5}

\bibitem[SZ13]{sturmfels2013binary}
  B. Sturmfels and P. Zwiernik.
  \emph{Binary cumulant varieties},
  Ann. Comb. 17.1 (2013), pp. 229--250.
  doi: \href{http://dx.doi.org/10.1007/s00026-012-0174-1}
  {10.1007/s00026-012-0174-1}

\bibitem[Wei92]{weispfenning1992comprehensive}
  V. Weispfenning.
  \emph{Comprehensive Gr\"obner bases},
  J. Symb. Comput. 14.1 (1992), pp. 1--29.
  doi: \href{http://dx.doi.org/10.1016/0747-7171(92)90023-W}
  {10.1016/0747-7171(92)90023-W}

\bibitem[Zak93]{Zak}
F. L. Zak.
  \emph{Tangents and secants of algebraic varieties}.
  Vol. 127. Translations of Mathematical Monographs.
  Translated from the Russian manuscript by the author.
  American Mathematical Society, Providence, RI, 1993.

\bibitem[Zwi15]{zwiernik2015semialgebraic}
  P. Zwiernik.
  \emph{Semialgebraic Statistics and Latent Tree Models}.
  CRC Press, 2015.

\end{thebibliography}
\end{document}